\documentclass[a4paper,reqno,11pt]{amsart}

\usepackage{amsmath,amsthm,amssymb,mathtools,url,breqn}

\usepackage{amssymb}
\usepackage{amstext}
\usepackage{amsmath}
\usepackage{amscd}
\usepackage{amsthm}
\usepackage{amsfonts}
\usepackage{enumerate}
\usepackage{graphicx}
\usepackage{color}
\usepackage{here} 
\usepackage{bm} 
\usepackage{mathrsfs}
\usepackage{mathtools}
\usepackage{latexsym}
\usepackage{booktabs}
\usepackage{fancyhdr}
\usepackage{subcaption}
\usepackage{tikz}
\usepackage{tikz-cd}
\usetikzlibrary{arrows}
\usetikzlibrary{decorations.markings}
\usetikzlibrary{patterns,decorations.pathreplacing}

\newcommand{\p}{\mathfrak{p}}

\newcommand{\Hom}{\textrm{Hom}}

\newcommand{\pd}{\textrm{pd}}

\newcommand{\dm}{\textrm{dim}}

\newcommand{\Spec}{\textrm{Spec}}

\newcommand{\depth}{\textrm{depth}}

\newcommand{\link}{\mathrm{link}}

\newcommand{\tr}{\textrm{tr}}
\newcommand{\bfa}{\mathbf{a}}
\newcommand{\bfb}{\mathbf{b}}
\newcommand{\bfv}{\mathbf{v}}
\newcommand{\bfu}{\mathbf{u}}

\newcommand{\bfx}{\mathbf{x}}

\newcommand{\ZZ}{\mathbb{Z}}

\newcommand{\QQ}{\mathbb{Q}}
\newcommand{\RR}{\mathbb{R}}

\newcommand{\kk}{\Bbbk}

\theoremstyle{definition}
\newtheorem{dfn}{Definition}[section]
\newtheorem{rem}[dfn]{Remark}
\newtheorem{fact}[dfn]{Fact}
\newtheorem{ex}[dfn]{Example}

\makeatletter
\@namedef{subjclassname@2020}{%
  \textup{2020} Mathematics Subject Classification}
\makeatother

\theoremstyle{plain}
\newtheorem{Q}[dfn]{Question}

\newtheorem{thm}[dfn]{Theorem}
\newtheorem{lem}[dfn]{Lemma}
\newtheorem{prop}[dfn]{Proposition}

\begin{document}

\title{Levelness versus nearly Gorensteinness of homogeneous domains}
\author[S. Miyashita]{Sora Miyashita}
\address[S. Miyashita]{Department of Pure And Applied Mathematics, Graduate School Of Information Science And Technology, Osaka University, Suita, Osaka 565-0871, Japan}
\email{miyashita.sora@ist.osaka-u.ac.jp}

\maketitle

\begin{abstract}
Levelness and nearly Gorensteinness are well-studied properties
of graded rings as a generalized notion of Gorensteinness.
In this paper, we compare the strength of these properties.
For any Cohen-Macaulay homogeneous affine semigroup
ring $R$,
we give a necessary condition for $R$
to be non-Gorenstein and nearly Gorenstein in terms of the $h$-vector of
$R$ and we show that
if $R$ is nearly Gorenstein with Cohen-Macaulay type 2,
then it is level.
We also show that if Cohen-Macaulay type is more than 2,
there are 2-dimensional counterexamples.
Moreover, we characterize nearly Gorensteinness of Stanley-Reisner rings of low-dimensional simplicial complexes.
\end{abstract}

\section{Introduction}
We denote the set of nonnegative integers, the set of integers, the set of rational numbers
and the set of real numbers
by $\mathbb{N}$, $\mathbb{Z}$, $\mathbb{Q}$ and $\mathbb{R}$, respectively.
Let $\kk$ be a field,
and let $R$ be an $\mathbb{N}$-graded $\kk$-algebra
with a unique graded maximal ideal $\mathbf{m}$.
We will always assume that $R$ is Cohen-Macaulay and admits a canonical module $\omega_R$.

Cohen-Macaulay (local or graded) rings and Gorenstein (local or graded)
rings are very important properties
and play a crucial role in the theory of commutative algebras.
For the study of a new class of local or graded rings which are Cohen-Macaulay
but not Gorenstein, many kinds of rings are defined.
For example, there are \textit{almost Gorenstein} rings, \textit{nearly Gorenstein} rings, and
\textit{level} rings.
There are also \textit{generically Gorenstein} rings and \textit{pseudo-Gorenstein} rings (see \cite[Proposition 3.3.18]{BH} and \cite{Pseudo}, respectively).

Originally, 
the notion of almost Gorenstein local rings of dimension one was introduced
by Barucci and Fr\"oberg \cite{Alm1} in the case where the local rings are
analytically unramified.
After this work, the general theory of almost Gorenstein rings
was introduced by Goto, Matsuoka, Phuong, Takahashi, and Taniguchi (see \cite{Alm2} and \cite{Alm3}).
In addition, Matsuoka and Murai \cite{MM} have studied
almost Gorenstein Stanley-Reisner rings.
For example, for a $1$-dimensional simplicial complex $\Delta$,
it was proved that $\Delta$ is an almost Gorenstein simplicial complex if and only if
$\Delta$ is a tree or a ridge sum of cycles
(see \cite[Proposition 3.8]{MM}).
On the other hand, $h$-vectors of almost Gorenstein rings are also studied by Higashitani in \cite{Alm4}.
There is a sufficient condition \cite[Theorem 3.1]{Alm4} to be almost Gorenstein in terms of $h$-vectors.

Nearly Gorenstein rings are a particularly new class of rings,
first defined by Herzog, Hibi, and Stamate in \cite{NG}.
Characterizations of nearly Gorenstein rings are known for
numerical semigroup rings, Hibi rings, edge rings of complete multipartite graphs, and
Ehrhart rings, and other things (see \cite{Trace,NG,H2,Miy1}).

Level rings were defined by Stanley (see \cite{BH} and \cite{Sta}).
It is well known that every doubly Cohen-Macaulay complex is level (see \cite[Theorem 5.7.6]{BH}).
There are also papers discussing levelness in ASL and Hibi rings (see \cite{H1} and \cite{Miy2}).

As we can see, there are many scattered notions of generalizations of Gorenstein rings.
Therefore, it is natural to compare the strength of these properties.
According to previous studies,
it is known that a 1-dimensional almost Gorenstein ring
is a nearly Gorenstein ring (see \cite[Proposition 6.1]{NG}),
and the characterization of levelness and almost Gorensteinness
of edge rings of complete multipartite graphs are given (see \cite{HM}).
In this paper, for the further
contribution,
we discuss how nearly Gorensteinness can lead to levelness under what conditions on the ring.

In Section 3, we consider some questions related to the comparison of
nearly Gorensteinness and levelness, and present many examples.
First, in the case of Cohen-Macaulay rings $R$ in general,
there are examples of rings that are nearly Gorenstein but not level (Example 3.1).
Even if we require the Cohen-Macaulay rings to be homogeneous domains,
we can find an example of a nearly Gorenstein ring, which is neither level, nor an affine semigroup ring
(Example 3.2).
These considerations lead us to the following question.

\begin{Q}
Let $R$ be a Cohen-Macaulay homogeneous affine semigroup ring.
If $R$ is nearly Gorenstein, then is $R$ level?
\end{Q}

We rephrase nearly Gorensteinness in terms of homogeneous affine semigroups,
and we answer Question 1.1.
First, we prove the following Theorem.

\newtheorem*{Mm}{\rm\bf Theorem~\ref{A1}}
\begin{Mm}
  Let $R$ be a Cohen-Macaulay homogeneous affine semigroup ring
  and $(h_0, h_1, . . . , h_s)$ its $h$-vector.
  If $R$ is nearly Gorenstein but not Gorenstein,
  then $h_s \geqq 2$.
\end{Mm}

From Theorem 4.4, we can prove next Theorem directly.

\newtheorem*{MainTheorem}{\rm\bf Theorem~\ref{A}}
\begin{MainTheorem}
  For any homogeneous affine semigroup ring, if it is nearly Gorenstein with Cohen-Macaulay type 2, then it is level.
\end{MainTheorem}

That is, in the case of type 2,
we resolved Question 1.1 in the affirmative.
On the other hand,
if Cohen-Macaulay type is more than 2,
there exist counterexamples to Question 1.1.

\newtheorem*{MainTheorem2}{\rm\bf Theorem~\ref{A2}}
\begin{MainTheorem2}
  For every $3\leq d \leq 5$,
  there exists type $d$ non-level nearly Gorenstein homogeneous affine semigroup ring.
\end{MainTheorem2}
Moreover, we discuss Hilbert series of
nearly Gorenstein homogeneous affine semigroup rings
with projective dimension and Cohen-Macaulay type
both equal to 2.
We show the Hilbert Series of $R$ has nice form
in this case.
 \newtheorem*{MTheorem}{\rm\bf Theorem~\ref{A3}}
 \begin{MTheorem}
  Let $d \geqq 2$ and let $R$ be a $d$-dimensional homogeneous affine semigroup.
  If $R$ is nearly Gorenstein with
  projective dimension and Cohen-Macaulay type both equal to 2
, then
\[H(R,t)=\frac{1+2\sum_{i=1}^s{t^i}}{(1-t)^d}.\]
 \end{MTheorem}
Lastly, we also discuss Stanley-Reisner rings of low-dimensional simplicial complexes.
\newtheorem*{ainTheorem}{\rm\bf Theorem~\ref{C}}
\begin{ainTheorem}

(a) Let $\Delta$ be a 0-dimensional simplicial complex
and let $R=\kk[\Delta]$ be the Stanley-Reisner ring of $\Delta$.
Then $R$ is nearly Gorenstein and level.

\noindent(b) Let $\Delta$ be a 1-dimensional connected simplicial complex, and
let $V$ be the set of vertices of $\Delta$.
Then the following conditions are equivalent:

(1) $\Delta$ is nearly Gorenstein;

(2) $\Delta$ is Gorenstein on the punctured spectrum;

(3) $\Delta$ is locally Gorenstein
(i.e., $\kk[\link_\Delta(\{i\})]$ is Gorenstein for every $i \in V$);

(4) $\Delta$ is a path or a cycle.

\noindent(c) Every 1-dimensional nearly Gorenstein simplicial complex is almost Gorenstein and level.
\end{ainTheorem}

The structure of this paper is as follows.
In Section 2, we prepare some definitions and facts for the discussions later.
In Section 3, we give some examples about the comparison of nearly Gorensteinness and levelness.
In Section 4, we show that
for any homogeneous affine semigroup ring,
if it is nearly Gorenstein with Cohen-Macaulay type 2,
then it is level.
We also show that if Cohen-Macaulay type is more than 2,
there are 2-dimensional counterexamples.
Moreover, we discuss the Hilbert series of nearly Gorenstein homogeneous affine semigroup rings.
In Section 5, we attempt to obtain results in the case of non-domains,
especially for squarefree monomial ideals.
We characterize nearly Gorensteinness of Stanley-Reisner rings of low-dimensional simplicial complexes.

\subsection*{Acknowledgement}
I am grateful to professor
Akihiro Higashitani for his very
helpful comments and instructive discussions.
I also thank Max K\"olbl
for his help in implementing
the program to find important counterexamples.

\section{Preliminaries}
Let $\kk$ be a field,
and let $R$ be an $\mathbb{N}$-graded $\kk$-algebra
with a unique graded maximal ideal $\mathbf{m}$.
We will always assume that $R$ is Cohen-Macaulay and admits a canonical module $\omega_R$.
\begin{itemize}
\item For a graded $R$-module $M$, we use the following notation:

― Fix an integer $k$.
Let $M(-k)$ denote the $R$-module whose grading is given by $M(-k)_n=M_{n-k}$
for any $n \in \mathbb{Z}$.

Moreover, if $k>0$, we write $M^{\oplus k} = M \oplus M \oplus \cdots \oplus M$ ($k$ times).

― Let $\tr_R(M)$ be the sum of the ideals $\phi(M)$
with $\phi \in \Hom_R(M,R)$. Thus,
\[\tr_R(M)=\sum_{\phi \in \Hom_R(M,R)}\phi(M).\]
When there is no risk of
confusion about the ring we simply write $\tr(M).$
\end{itemize}

\begin{itemize}
  \item Let $r(R)$ be the Cohen-Macaulay type of $R$,
  and let $\pd(R)$ be the projective dimension of $R$.
    \end{itemize}
  
  \begin{itemize}
      \item Let $|X|$ be the cardinality of the set $X$.
  \end{itemize}

Let us recall the definitions and facts of the nearly Gorensteinness and levelness
of graded rings.

\begin{dfn}[{see \cite[Chapter III, Proposition 3.2]{Sta}}]
We say that $R$ is\;$\textit{level}$\: if all the degrees of the minimal
generators of $\omega_R$ are the same.
\end{dfn}

\begin{dfn}[{see \cite[Definition 2.2]{NG}}]
We say that $R$ is $\textit{nearly Gorenstein}$ if $\tr(\omega_R) \supseteq \mathbf{m}$.
In particular,  $R$ is nearly Gorenstein but not Gorenstein
if and only if $\tr(\omega_R) = \mathbf{m}$.
\end{dfn}

Gorenstein on the punctured spectrum is also defined here, as it will be needed later.

\begin{dfn}[{see \cite[Proposition 3.3.18]{BH}}]
We say that $R$ is $\textit{Gorenstein on the punctured}$
$\textit{spectrum}$
if $R_\p$ is Gorenstein for any $\p \in \Spec(R)\setminus \{\mathbf{m} \}$.
\end{dfn}
\begin{rem}
It is known that
$R$ is Gorenstein on the punctured spectrum if and only if $\tr(\omega_R) \supseteq \mathbf{m}^k$ for some $k \in \mathbb{N}$,
where $\mathbf{m}^0=R$ (see \cite[Proposition 2.3]{NG}).
In particular, if $R$ is nearly Gorenstein, then $R$ is Gorenstein on the punctured spectrum.
\end{rem}

Let $R$ be a ring and $I$ an ideal of $R$ containing
a non-zero divisor of $R$. Let $Q(R)$ be the total
quotient ring of fractions of $R$ and set
$\displaystyle I^{-1}:=\{x \in Q(R):xI\subset R\}.$ Then
\begin{equation}
\displaystyle \tr(I)=II^{-1}
\end{equation}
(see \cite[Lemma 1.1]{NG}).

If $R$ is an $\mathbb{N}$-graded ring, then
$\omega_R$ is isomorphic to an ideal $I_R$ of $R$ as an $\mathbb{N}$-graded module
up to degree shift if and only if $R_{\mathfrak{p}}$ is Gorenstein
for every minimal prime ideal $\mathfrak{p}$
(for example, if $R$ is a domain or a Stanley-Reisner ring).
We call $I_R$ the canonical ideal of $R$.

\begin{fact}[{see \cite[Chapter I, Section 12]{Sta}}]
Fix an integer with $n \geq  2$. If $R$ is an $\mathbb{N}^n$-graded domain, then
$\omega_R$ is isomorphic to an ideal $I_R$ of $R$ as an $\mathbb{N}^n$-graded module
up to degree shift.
\end{fact}

We recall some definitions about affine semigroups.
An \textit{affine semigroup} $S$ is a finitely generated sub-semigroup of $\mathbb{N}^d.$
Moreover, we say that $S$ is \textit{homogeneous} if all its minimal generators lie on an affine hyperplane not including origin.
Equivalently, the affine semigroup ring $\kk[S]$ is standard graded by assigning degree one
to all the monomials corresponding to the minimal generators of $S$.
In that case, we also say that $\kk[S]$ is \textit{homogeneous}.
We denote the group generated by $S$ by $\ZZ S$,
the convex cone generated by $S$ by $\RR_{\geqq 0} S\subseteq \RR^d$
and the normalization by $\overline{S}=\ZZ{S} \cap \RR_{\geqq 0} S$.
The affine semigroup $S$ is pointed
if $S\cap(-S)=\{0\}$.
We can check easily that homogeneous affine semigroup $S$ is pointed.
It is known that $S$ is pointed if and only if
the associated cone $C=\RR_{\geqq0} S$ is pointed ({see \cite[Lemma 7.12]{MS}}).
Moreover, every pointed affine semigroup $S$ has
a unique finite minimal generating set ({see \cite[Proposition 7.15]{MS}}).
Thus if $S$ is pointed, $C=\RR_{\geqq0} S$ is a finitely generated cone.
A face $F \subseteq S$ of $S$ is a subset such that
for every $a,b \in S$ the following holds:
\[a+b \in F \Leftrightarrow a \in F \text{\;and\;} b \in F.\]
The 1-dimensional faces of a pointed semigroup $S$ are called its extremal rays.
We prepare the following basic lemma for proving Theorem 4.4.
We denote $(\bfa,\bfb)$ as inner product of $\bfa, \bfb \in \RR^d$.
\begin{lem}
Let $d \geq 2$ and let $S$ be a $d$-dimensional
pointed affine semigroup, and let $C=\RR_{\geqq 0} S$.
Let $E$ be the set of extremal ray of $C$.
If $\bfx \notin C$,
then there exists $l \in E$
such that $(\bfx + l) \cap C = \emptyset$.
\begin{proof}
  Let $\{F_1,\cdots,F_n\}$ be
  the set of all facets of $C$.
  We can write $C=\bigcap_{i=1}^n
  \{\bfx \in \RR^d ; (\bfa_{F_i},\bfx) \geqq 0 \}$
  where $\bfa_{F_i} \in \RR^d$
  for all $1 \leq i \leq n$.
  Since $\bfx \notin C$,
  there exists a facet $F_i$
  such that $(\bfa_{F_i},\bfx)<0$.
  For all $\bfa \in F_i$,
  since $(\bfa_{F_i}, \bfx+\bfa)=(\bfa_{F_i}, \bfx)<0$,
  we get $\bfx+\bfa \notin C$.
  Thus $(\bfx + F_i) \cap C = \emptyset$.
  In particular,
  we choose a 1-dimensional face $l$ of $F_i$.
  Then $l \in E$ and
  $(\bfx + l) \cap C = \emptyset$.
  \end{proof}
\end{lem}
We recall some important facts.

\begin{thm}[{see \cite[Theorem 3.1]{Lk}}]
  Let $S$ be a pointed affine semigroup.
  There exists a (not necessarily disjoint) decomposition
  \begin{equation}\overline{S} \setminus S = \bigcup_{i=1}^l(s_i+\ZZ F_i) \cap \RR_{\geq 0}S
  \end{equation}
  with $s_i \in \overline{S}$
  and faces $F_i$ of $S$.
  \end{thm}

  A set
  $s_i+\ZZ F_i$ from (2) of called
  a $j$-dimensional family of holes,
  where $j$ is the dimension of $F_i$.
  
  \begin{thm}[{see \cite[Theorem 5.2]{Lk}}]
  Let $S$ be a pointed affine semigroup of dimension $d$.
Then $\kk[S]$ satisfies Serre's condition $(S_2)$
if and only if
every family of holes has dimension $d-1$.
    \end{thm}

\begin{thm}[{see \cite[Corollary 3.2 and Corollary 3.5]{NG}}]
  Let $S=\kk[x_1,\cdots,x_n]$ be a polynomial ring, let $\mathbf{n}=(x_1,\cdots,x_n)$ be the graded maximal ideal of $S$ and let
  \[\mathbb{F}:0 \rightarrow F_p \xrightarrow{\phi_p} F_{p-1} \rightarrow \cdots \rightarrow F_1 \rightarrow F_0 \rightarrow R \rightarrow 0\]
  be a graded minimal free $S$-resolution of the Cohen-Macaulay ring $R=S/J$ with
  $J \subseteq \mathbf{n}^2$.
  Let $I_1(\phi_p)$ be an ideal of $R$ generated by all components of
  representation matrix of $\phi_p$.
  Then the following holds.

  (a)  Let $e_1,\cdots,e_t$ be a basis of $F_p$.
  Suppose that for $i=1,\cdots,s$
  the elements $\sum_{j=1}^t r_{ij}e_j$ generate
  the kernel of
  \[\psi_p \;;\; F_p \otimes R \longrightarrow F_{p-1} \otimes R,\]
  where
  \[\psi_p=\phi_p \otimes R.\]
  Then $\tr(\omega_R)$ is generated by the elements $r_{ij}$ with $i=1,\cdots,s$ and $j=1,\cdots,t$.

  (b) If $r(R)=2$ and
  $\dm R>0$,
  then $I_1(\phi_p)=\mathbf{n}$ if and only if $R$ is nearly Gorenstein.
  \end{thm}

We state the necessary results about the minimal free resolution
of the codimension 2 lattice ideal based on \cite{IP}.

\begin{dfn}
Let $S=\kk[x_1,\cdots,x_n]$ be a polynomial ring and let
$L$ be any sublattice of $\mathbb{Z}$.
We put
$\bfx^{\bfa} := {x_1}^{a_1}{x_2}^{a_2}\cdots {x_n}^{a_n}$
where $\bfa=(a_1,a_2, \cdots, a_n) \in \mathbb{N}^n.$
Then its associated \textit{lattice ideal} in $S$ is
\[I_L := (\bfx^\bfa-\bfx^\bfb \;;\; \bfa,\bfb \in \mathbb{N}^n \;\;\textit{and} \;\;\bfa-\bfb \in L ).\]
Prime lattice ideals are called \textit{toric ideals}.
Prime binomial ideals and toric ideals are identical ({see \cite[Theorem 7.4]{MS}}).
\end{dfn}

\begin{prop}[{see \cite[Comments 5.9 (a) and Theorem 6.1 (ii)]{IP}}]
Let $S=\kk[x_1,\cdots,x_n]$ be a polynomial ring.
If $I$ is a codimension $2$ lattice ideal of $S$
and the number of minimal generators of $I$
is 3, then
$R=S/I$ is Cohen-Macaulay and
the graded minimal free resolution of $R$ is the following form.
\[0\rightarrow S^{2} \xrightarrow {\left[
\begin{array}{cc}
u_1 & u_4 \\
u_2 & u_5 \\
u_3 & u_6
\end{array}
\right]
}
S^{3} \rightarrow
 S \rightarrow R \rightarrow 0,\]
where $u_i$ is a monomial of $S$ for all $1 \leq i \leq 6$.
\end{prop}
Note that a codimension 2 prime binomial ideal $I$ is Cohen-Macaulay
but not Gorenstein if and only if the number of
minimal generators of $I$ is 3 ({see \cite[Remark 5.8 and Theorem 6.1]{IP}}).
\section{Examples: nearly Gorensteinness versus levelness}
The following example show that in the case of non-domains,
nearly Gorensteinness does not imply levelness.
\begin{ex}
Let $S=\QQ[x,y,z]$ be a graded polynomial ring with
$\deg x = \deg y=\deg z=1.$
Consider a homogeneous ideal $I=(xz,yz,y^3)$ and define $R=S/I$,
then the graded minimal free resolution of $R$ is as follows.
\[0\rightarrow S(-3) \oplus S(-4) \xrightarrow {\left[
\begin{array}{cc}
-y & 0 \\
x & -y^2 \\
0 & z
\end{array}
\right]
}
S(-2)^{\oplus 2} \oplus S(-3) \rightarrow
 S \rightarrow R \rightarrow 0.\]
Thus $r(R)=2$ and $R$ is not level,
and $R$ is Cohen-Macaulay because $\dm R = \depth R = 1 > 0$.
Then, $R$ is nearly Gorenstein by Theorem $2.9$.
\end{ex}

%

Next, we consider the case of the Cohen-Macaulay homogeneous domain.
Even in that case, we can find the following example.

\begin{ex}
  Let $S=\QQ[x,y,z]$ be a graded polynomial ring with
  $\deg x=\deg y=\deg z=1.$
  Consider a homogeneous prime ideal $P=(xz^2-y^3,x^3+xy^2-y^2z,x^2y+y^3-z^3)$ and define $A=S/P$,
  then the graded minimal free resolution of $A$ is as follows.
  
  \[0\rightarrow S(-4) \oplus S(-5) \xrightarrow {\left[
  \begin{array}{cc}
  -y & z^2 \\
  x & -y^2 \\
  -z & x^2+y^2
  \end{array}
  \right]
  }
  S(-3)^{\oplus 3} \rightarrow
   S \rightarrow A \rightarrow 0.\]
  
  Thus $r(A)=2$ and $A$ is not level
  but nearly Gorenstein by Theorem $2.9$.
  \end{ex}




Since $P$ is not a toric ideal, this counterexample is not an affine semigroup ring.
Affine semigroup rings which are Gorenstein on the punctured spectrum (see Remark 2.4)
are not necessarily level.

\begin{ex}
Let $S=\QQ[x_1,x_2,x_3,x_4,x_5,x_6]$ be a standard graded polynomial ring.
Consider a homogeneous toric ideal $P=(x_1{x_5}^2-x_2{x_4}^2,x_1{x_6}^2-x_3{x_4}^2,x_2{x_6}^2-x_3{x_5}^2)$
and define $R=S/P$,
then the graded minimal free resolution of $R$ is as follows.

\[0\rightarrow S(-4) \oplus S(-5) \xrightarrow {\left[
\begin{array}{cc}
x_1 & -{x_4}^2 \\
-x_2 & {x_5}^2 \\
x_3 & -{x_6}^2
\end{array}
\right]
}
S(-3)^{\oplus 3} \rightarrow
 S \rightarrow R \rightarrow 0.\]

Thus $R$ is not level but $\tr(\omega_R)\supseteq \mathbf{m}^4$ by Theorem $2.9$.
Then $R$ is Gorenstein on the punctured spectrum.
\end{ex}
\section{nearly Gorenstein homogeneous affine semigroup rings}
We rephrase the condition of nearly Gorensteinness in terms of semigroups.
First, we prepare the next lemma.

\begin{lem}
Let $R$ be an $\mathbb{N}^n$-graded domain,
and let $p$ be a homogeneous element of $R$, $0 \neq q \in R$.
We put a non-zero ideal $I_R = (f_1, \cdots , f_r)$,
and let \[q=\sum_{\substack{\bfa \in \mathbb{N}^n \\ q_\bfa \neq 0}}q_\bfa\] be
an $\mathbb{N}^n$-graded decomposition of $q$.
Then $\displaystyle \frac{q}{p} \in I_R^{-1}$ if and only if $\displaystyle \frac{q_\bfa}{p} \in I_R^{-1}$
for any $\bfa \in \mathbb{N}^n$.
\begin{proof}
The ``if part" is obvious.
We show ``only if" part.
Assume that there exists some
$\bfa \in \mathbb{N}^n$ such that $\displaystyle 0 \neq \frac{q_\bfa}{p} \in I_R^{-1}$,
then we get

\[\frac{q}{p}=\sum_{\substack{\bfa \in \mathbb{N}^n \\ 0 \neq \frac{q_\bfa}{p} \notin I_R^{-1}}}{\frac{q_\bfa}{p}}+\sum_{\substack{\bfb \in \mathbb{N}^n \\ 0 \neq \frac{q_\bfb}{p} \in I_R^{-1}}}{\frac{q_\bfb}{p}} \in I_R^{-1}.\]

Therefore, $\displaystyle \sum_{\substack{a \in \mathbb{N}^n \\ 0 \neq \frac{q_\bfa}{p} \notin I_R^{-1}} }{\frac{q_\bfa}{p}} \in I_R^{-1}.$
Thus for any $1 \leq i \leq r$, there exists an $x_i \in R$ such that
\begin{equation}
\sum_{\substack{\bfa \in \mathbb{N}^n \\0 \neq \frac{q_\bfa}{p} \notin I_R^{-1}}}{q_\bfa f_i} = px_i.
\end{equation}
Fix $\bfa \in \mathbb{N}^{n}$ such that $\displaystyle 0 \neq \frac{q_\bfa}{p} \notin I_R^{-1}$,
and compare the degree $(\bfa+\deg f_i)$ of both sides of equality $(3)$,
then we get $\displaystyle q_\bfa f_i = p(x_i)_{\bfa+\deg{f_i}-\deg p}$. Thus $\displaystyle \frac{q_\bfa}{p}f_i \in R$ for any $1 \leq i \leq r$.
Therefore, we get $\displaystyle \frac{q_a}{p} \in I_R^{-1}$, a contradiction.
\end{proof}
\end{lem}
Let $S$ be a Cohen-Macaulay homogeneous affine semigroup,
and let $G_S=\{ \bfa_1,\cdots,\bfa_s \} \subseteq \mathbb{N}^n$ be the minimal generators of $S$.
Fix the affine semigroup ring $R=\kk[S]$.
For any $\bfa \in \mathbb{N}^n$, we set
\begin{equation} \notag
R_\bfa
=
\begin{cases}
\kk \bfx^\bfa & (\bfa \in S)\\
0 & (\bfa \notin S).
\end{cases}
\end{equation}
Then $R = \bigoplus_{\bfa \in \mathbb{N}^n }R_\bfa$ is a direct sum decomposition as an abelian group,
and $R_\bfa R_\bfb \subseteq R_{\bfa+\bfb}$ for any $\bfa,\bfb \in \mathbb{N}^n$.
Thus we can regard $R$ as an $\mathbb{N}^n$-graded ring.
Since $S$ is homogeneous, we can regard
$R=\kk[S]=\kk[\bfx^{\bfa_1},\cdots,\bfx^{\bfa_s}]$ as standard graded by assigning
$\deg \bfx^{\bfa_i}=1$
for all $1 \leq i \leq s$.
In this case, the canonical module $\omega_{R}$ is isomorphic to an ideal $I_{R}$ of $R$
as an $\mathbb{N}^n$-graded module up to degree shift.
Then we can assume the minimal generators of $I_R$
is $\{ \bfx^{\bfv_1},\cdots,\bfx^{\bfv_r} \}$,
and $V=\{\bfv_1,\cdots,\bfv_r\} \subseteq S$ is a minimal generator of canonical ideal of $S$.
We put
$V_{\min}
=
\{\bfv \in V
; \deg \bfx^\bfv \leqq
\deg \bfx^{\bfv_i} \;\text{for all}
\; 1 \leqq i \leqq r\}$
and
$S-V
:= \{ \bfa \in \mathbb{Z}S \;;\; \bfa + \bfv \in S \; \text{for all}\; \bfv \in V\}$.
Thus the following holds.

\begin{prop}
  Let $S$ be a Cohen-Macaulay
  homogeneous affine semigroup.
  The following are equivalent:
  
  (a) $R=\kk[S]$ is nearly Gorenstein;
  
  (b) For any $\bfa_i \in G_S$,
  there exist $\bfv \in V$
  and $\bfu \in S-V$ such that
  $\bfa_i = \bfu + \bfv$;
  
  (c) For any $\bfa_i \in G_S$,
  there exist $\bfv \in V_{\min}$
  and $\bfu \in S-V$
  such that
  $\bfa_i = \bfu + \bfv$.
\begin{proof}
Since $S-V=\{ \bfa \in \mathbb{Z}S \;;\; \bfx^{\bfa} \in {I_R}^{-1} \}$,
$(b) \Rightarrow (a)$ is obvious.
We show $(a) \Rightarrow (b)$.
Let $R$ be nearly Gorenstein.
For all $1 \leq i \leq m$,
we know from equality (1)
that ${\bfx}^{\bfa_i} \in I_RI_R^{-1}$.
Thus, there exists $g_{ik} \in I_R^{-1}$ for all $1 \leq k \leq r$
such that
${\bfx}^{\bfa_i} = \sum_{k=1}^r{ g_{ik}{\bfx}^{\bfv_k} }.$
We can write $g_{ik} = \frac{ p_{ik} }{ \bfx^{\bfv_1} }$
where $p_{ik} \in R$.
Moreover, we consider the $\mathbb{N}^n$-graded decomposition of $p_{ik}$.
Since $g_{ik} \in I_R^{-1}$,
by using Lemma 4,1, there exist $l_{1i},\cdots,l_{ri} \in \mathbb{N}$
and a set of monomials
$Z_i=\{u_{1,1},\cdots,u_{1,l_{1i}},u_{2,1},\cdots,u_{2,l_{2i}},\cdots,
u_{r,1},\cdots,u_{r,l_{ri}}\} \subset I_R^{-1}$
that also allow negative powers and satisfy the following equality:
 \[\bfx^{\bfa_{i}}=\sum_{k=1}^r\sum_{j=1}^{l_{ki}}u_{k,j}\bfx^{\bfv_k}.\]
By comparing the degree $\bfa_i$ of both sides of
the equality,
there exist $1 \leq n_i \leq r$,
and $u_i \in Z$ such that $\bfx^{\bfa_{i}}=u_i \bfx^{\bfv_{n_i}}$.
Since $u_i \in {I_R}^{-1}$, there exists $\bfu_i \in S-V$ such that
$u_i=\bfx^{\bfu_i}$.
Then $\bfx^{\bfa_{i}}=\bfx^{\bfu_i+\bfv_{n_i}}$
and we get $\bfa_{i}=\bfu_i+\bfv_{n_i}$, as desired.

$(c) \Rightarrow (b)$ is obvious.
We show $(b) \Rightarrow (c)$.
If $R$ is level, then $V=V_{\min}$ and this
is true.
Assume $R$ is non-level.
By the assumption,
for any $1 \leqq i \leqq s$,
there exist $\bfv \in V$ and $\bfu \in S-V$
such that $\bfa_i=\bfu + \bfv$.
It is enough to show $\bfv \in V_{\min}$.
If $\bfv \notin V_{\min}$,
then there exists $\bfv' \in V_\min$
such that $\deg \bfx^{\bfv'-\bfv} <0$.
Since $\bfu \in S-V$,
we get
$\bfa_i + \bfv'-\bfv=\bfu + \bfv' \in S \setminus \{0\}$.
Thus, we have $\deg \bfx^{\bfv'-\bfv}=\deg \bfx^{\bfa_i+\bfv'-\bfv}-1 \geqq 0$,
which yields a contradiction.
\end{proof}
\end{prop}
We put
\[\tr(\omega_S)=\{\bfa \in G_S\;;\; \text{there exist }\bfv \in V_\min \;\text{and }\; \bfu \in S-V \text{such that } \bfa=\bfv+\bfu\}.\]
By Proposition 4.2, $R=\kk[S]$ is nearly Gorenstein
if and only if $\tr(\omega_S)=G_S$.
\begin{thm}
  Let $E$ be the set of extremal rays of $S$.
  If $V_{\min}=\{\bfv\}$ and $|V| \geqq 2$,
  then $E \nsubseteq \tr(\omega_S)$.
  \end{thm}
\begin{proof}
In this case,
we can write
$\tr(\omega_S)=
\{\bfa \in G_S\;;\;
\text{there exists }\bfu \in S-V \text{such that }
\bfa=\bfv+\bfu\}$.
We assume $E \subseteq \tr(\omega_S)$.
Take $\bfv' \in V$ such that $\deg \bfx^\bfv < \deg \bfx^{\bfv'}$.
Since $\bfv, \bfv' \in V$ and
$\bfv \neq \bfv'$,
we get $\bfv'-\bfv \in \ZZ{S}\setminus S$.
We assume
$\bfv'-\bfv \notin \RR_{\geqq 0}S$.
Then by Lemma 2.6,
there exists $\bfa \in E$
such that
$\bfv' - \bfv + \bfa \notin S.$
On the other hand,
since $\bfa \in \tr(\omega_S)$,
there exists $\bfu \in S-V$
such that $\bfa=\bfv+\bfu$.
Thus we get $\bfv' - \bfv + \bfa \in S,$
this yields a contradiction.

Then, $\bfv'-\bfv \in \RR_{\geqq 0}S$ and
we get $\bfv'-\bfv \in \overline{S} \setminus S$.
Since $S$ is Cohen-Macaulay, 
$S$ satisfies $(S_2)$-condition.
By Theorem 2.8,
every family of holes has dimension $d-1$.
    Since $\bfv'-\bfv \in \overline{S}\setminus S$,
there exist
$s_i \in \overline{S}$ and
facet $F$ of $S$ such that
$\bfv'-\bfv \in s_i +\ZZ F$
and $(s_i+\ZZ F) \cap S= \emptyset$
by using Theorem 2.7.
Since $\bfv'-\bfv \in s_i +\ZZ F$,
we can take $\bfx \in \ZZ F$ and write
$\bfv'-\bfv=s_i+\bfx$.
Thus, we get
$(\bfv'-\bfv + \ZZ F)\cap S=(s_i+\ZZ F) \cap S=\emptyset$.
In particular,
pick an extremal ray $l$ of facet $F$,
we get
    ($\bfv'-\bfv + \ZZ l) \cap S = \emptyset$.
On the other hand,
we have $\bfv'-\bfv + l \in S$ because $l \in E$.
This yields a contradiction.
\end{proof}
\begin{thm}\label{A1}
  Let $(h_0, h_1, . . . , h_s)$ be the $h$-vector of $R$.
  If $R$ is nearly Gorenstein but not Gorenstein,
  then $h_s \geqq 2$.
  \end{thm}
  \begin{proof}
    If $h_s=|V_{\min}|=1$, then $E \nsubseteq \tr(\omega_S)$ by Theorem 4.3.
    Thus $R$ is not nearly Gorenstein. This gives a contradiction.
    \end{proof}
\begin{thm}\label{A}
For any homogeneous affine semigroup ring, if it is nearly Gorenstein with Cohen-Macaulay type 2, then it is level.
\end{thm}
\begin{proof}
If $R$ is not level, then $h_s=1$.
This contradicts Theorem 4.4.
\end{proof}
\begin{ex}
  Let $R=\QQ[s,st,st^3,st^4,st^{9},st^{14}]$
  be a Cohen-Macaulay homogeneous affine semigroup ring.
  Then $h$-vector of $R$ is $(1,4,8,1)$.
  Thus $R$ is not Gorenstein.
  Since $h_4=1$, $R$ is not nearly Gorenstein by Theorem 4.4.
\end{ex}
\begin{rem}
The same statement as Theorem 4.4 and Theorem 4.5 does not hold
for general homogeneous domains.
For example,
consider nearly Gorenstein and non-level
homogeneous domain $A$ in Example 3.2,
then $r(A)=2$ and
the $h$-vector of $A$ is $(1,2,3,1)$.
\end{rem}
Next, we give counterexamples to Question 1.1 if $r(R) \geqq 3.$
Before that, we prepare a little more.
For simplicial affine semigroup ring,
Cohen-Macaulayness is determined independently of field $\kk$ ({see \cite[Theorem (1)]{Aff1}}).
Moreover, for Cohen-Macaulay simplicial affine semigroup ring,
the canonical module is uniquely determined independently of field $\kk$ ({see \cite[Theorem (3)]{Aff1}}).
Then the next statement holds.
\begin{prop}
Let $S$ be a simplicial affine semigroup.
The following conditions are equivalent:

(1) $\QQ[S]$ is nearly Gorenstein;

(2) There exists a field $\kk$ such that $\kk[S]$ is nearly Gorenstein;

(3) For every field $\kk$, $\kk[S]$ is nearly Gorenstein.

The same statement holds if we change“nearly Gorenstein”to“level”.
\end{prop}

For homogeneous affine semigroup rings with type 3
or more, nearly Gorensteinness
does not imply levelness in general.

\begin{thm}\label{A2}
For every $3\leq d \leq 5$,
there exists type $d$ non-level nearly Gorenstein homogeneous affine semigroup ring.
\begin{proof}
Let $\kk$ be a field.
\begin{itemize}
  \item If $d=3$, it is enough to show that
  $R_3=\kk[s,st^2,st^4,st^5,st^{7},st^{9},st^{12},st^{17}]$
  is type 3 non-level nearly Gorenstein.
  By Proposition 4.7, we can assume $\kk=\QQ$.
Then the graded minimal free resolution of $R_3$ is as follows.
\begin{align}\nonumber
0 &\rightarrow S(-8) \oplus S(-9)^{\oplus 2}
\xrightarrow{A_{R_3}} S(-7)^{\oplus 12} \oplus S(-8)^{\oplus 5}
\rightarrow S(-5)^{\oplus 4} \oplus S(-6)^{\oplus 45}\\
&\rightarrow S(-4)^{\oplus 25} \oplus S(-5)^{\oplus 50}
\rightarrow S(-3)^{\oplus 30} \oplus S(-4)^{\oplus 28}\nonumber
\rightarrow S(-2)^{\oplus 13} \oplus S(-3)^{\oplus 6}\\
&\rightarrow S \rightarrow R_3 \rightarrow 0, \nonumber
\end{align}
where $S=\QQ[x_1,\cdots,x_8]$ is a polynomial ring.
Thus $R_3$ is Cohen-Macaulay and type 3 non-level.
By using $\mathtt{Macaulay2}$ ({\cite{M2}}),
one can see that
the generator of $\ker(R_3 \otimes_S A_{R_3})$ is as follows.

\[\left[
\begin{array}{cccccc}
{x_6}^2 \\
-x_8 \\
-x_7
\end{array}
\right],
\left[
\begin{array}{cccccc}
x_3x_6 \\
-x_7 \\
-x_5
\end{array}
\right],
\left[
\begin{array}{cccccc}
{x_4}^2 \\
-x_6 \\
-x_3
\end{array}
\right],
\left[
\begin{array}{cccccc}
{x_3}^2 \\
-x_5 \\
-x_2
\end{array}
\right],
\left[
\begin{array}{cccccc}
x_2x_3 \\
-x_4 \\
-x_1
\end{array}
\right]
\in {R_3}^3\]
Then $R_3$ is nearly Gorenstein by Theorem 2.9.
\item If $d=4$ or $5$, it is enough to show that
  $R_4=\kk[s,st^{4},st^{9},st^{12},st^{13},st^{21}]$
  is type 4 non-level nearly Gorenstein
  and
  $R_5=\kk[s,st^6,st^7,st^9,st^{13},st^{15},st^{19}]$
  is type 5 non-level nearly Gorenstein,
  respectively.
  In the same way as above, we can check
  $R_4$ and $R_5$ are non-level nearly Gorenstein,
  $r(R_4)=4$ and $r(R_5)=5$.
  Moreover, the generator of $\ker(R_4 \otimes_S A_{R_4})$ and $\ker(R_5 \otimes_S A_{R_5})$
  are as follows, respectively.
  \[\left[
  \begin{array}{cccccc}
  {x_4}^3 \\
  -x_6 \\
  x_5 \\
  x_3
  \end{array}
  \right],
  \left[
  \begin{array}{cccccc}
  {x_3}^3 \\
  -x_4 \\
  x_2 \\
  x_1
  \end{array}
  \right],
  \left[
  \begin{array}{cccccc}
  {x_1}^2x_2x_5{x_6}^2 \\
  -{x_3}^2{x_5}^2 \\
  x_1{x_4}^3 \\
  {x_2}^2{x_4}^2
  \end{array}
  \right]
  \in {R_4}^4\]
  \[\left[
  \begin{array}{cccccc}
  {x_4}^2 \\
  x_7 \\
  x_6 \\
  x_5 \\
  x_2
  \end{array}
  \right],
  \left[
  \begin{array}{cccccc}
  {x_2}^2 \\
  x_5 \\
  x_4 \\
  x_3 \\
  x_1
  \end{array}
  \right],
  \left[
  \begin{array}{cccccc}
  x_1x_3x_4x_7 \\
  x_2{x_6}^2 \\
  x_1x_5x_7 \\
  x_1{x_6}^2 \\
  {x_3}^2x_4
  \end{array}
  \right]
  \in {R_5}^5\]
  Then $R_4$ and $R_5$ are nearly Gorenstein by Theorem 2.9.
\end{itemize}
\end{proof}
\end{thm}
Lastly, we discuss Hilbert series of nearly Gorenstein homogeneous affine semigroup ring $R$
with $r(R)=\pd(R)=2$.
In this case,
we can determine all minimal graded free resolutions of $R$
and show that
the Hilbert series of $R$ has nice form.
We denote $|\bfa|=\sum_{k=1}^{n}a_i$
where $\bfa=(a_1,a_2, \cdots, a_n) \in \mathbb{N}^n.$
 \begin{thm}\label{A3}
    Let $d \geqq 2$ and let $R$ be a $d$-dimensional homogeneous affine semigroup.
    If $R$ is nearly Gorenstein and $\pd(R) = r(R) = 2$
    , then
    \[H(R,t)=\frac{1+2\sum_{i=1}^s{t^i}}{(1-t)^d}.\]
\begin{proof}
By the assumption,
there exists a codimension 2 homogeneous prime binomial ideal $I$
such that
$I$ is minimally generated by three elements and
$R \cong S/I$, where $S=\kk[x_1,\cdots,x_n]$ is a polynomial ring.
Since $I$ is a codimension $2$ lattice ideal and the number of minimal generators of $I$
is 3,
$R$ is a $(n-2)$-dimensional Cohen-Macaulay ring
and the graded minimal free resolution of $R$ is of the following form by Proposition 2.11.
\[0\rightarrow S^{2} \xrightarrow {A =\left[
\begin{array}{cc}
u_1 & -u_4 \\
-u_2 & u_5 \\
u_3 & -u_6
\end{array}
\right]
}
S^{3} \rightarrow
 S \rightarrow R \rightarrow 0.\]
Here, $u_i$ is a monomial of $S$ for all $1 \leq i \leq 6$.
By using Hilbert-Burch Theorem ({see \cite[Theorem 1.4.17]{BH}}), $I$ is minimally generated by
$f_1=u_1u_5-u_2u_4, f_2=u_3u_4-u_1u_6$ and $f_3=u_2u_6-u_3u_5$.
Since $I$ is a prime binomial ideal, for all $i=1,2,3$, there exist
$\bfa_i, \bfb_i \in \mathbb{N}^n$ such that
$f_i=\bfx^{\bfa_i}-\bfx^{\bfb_i}$, $|\bfa_i|=|\bfb_i|$
and $\gcd(\bfx^{\bfa_i},\bfx^{\bfb_i})=1$.
We assume that $R$ is nearly Gorenstein.
\begin{itemize}
\item If $d=2$, since $R$ is nearly Gorenstein,
$A$ may be assumed to have one of the following forms by Theorem 2.9.

(i) $A =\left[
\begin{array}{cc}
x_1 & -x_4 \\
-x_2 & u_5 \\
x_3 & -u_6
\end{array}
\right]
$
or
(ii) $A =\left[
\begin{array}{cc}
x_1 & -x_3 \\
-u_2 & x_4 \\
x_2 & -u_6
\end{array}
\right]
$
or
(iii) $A =\left[
\begin{array}{cc}
x_1 & -x_3 \\
-x_2 & x_4 \\
u_3 & -u_6
\end{array}
\right]
$.

(For example,
there is also
a possibility that
$A =\left[
\begin{array}{cc}
u_1 & -x_2 \\
-u_2 & x_3 \\
x_1 & -x_4
\end{array}
\right]
$,
but this
can be regarded to be
the same as (i).)

In the cases of (i) and (ii),
we see that all components of the matrix $A$ are variables $x_i$.
Then the graded minimal free resolution of $R$ is as follows.
\[0\rightarrow S(-3)^{\oplus 2} \rightarrow
S(-2)^{\oplus 3} \rightarrow
 S \rightarrow R \rightarrow 0.\]
Thus
$\displaystyle
H(R,t)=H(S,t)-3H(S(-2),t)+2H(S(-3),t)=\frac{1+2t}{(1-t)^2}$.

(iii) Assume the case of $A =\left[
\begin{array}{cc}
x_1 & -x_3 \\
-x_2 & x_4 \\
u_3 & -u_6
\end{array}
\right]
$.
We can write $u_3={x_3}^{n_1}{x_4}^{m_1}, u_6={x_1}^{n_2}{x_2}^{m_2}$
where $N=n_1+m_1=n_2+m_2, (0,0) \neq (n_i,m_i) \in \mathbb{N}^2$ for each $i=1,2$.
Then the graded minimal free resolution of $R$ is as follows.
\[0\rightarrow S(-N-2)^{\oplus 2} \xrightarrow {\left[
\begin{array}{cc}
x_1 & -x_3 \\
-x_2 & x_4 \\
{x_3}^{n_1}{x_4}^{m_1} & -{x_1}^{n_2}{x_2}^{m_2}
\end{array}
\right]
}
S(-N-1)^{\oplus 2} \oplus S(-2) \rightarrow
 S \rightarrow R \rightarrow 0.\]
 Thus $\displaystyle
 H(R,t)=\frac{1-(2t^{N+1}+t^2)+2t^{N+2}}{(1-t)^4}
 =\frac{1+2\sum_{i=1}^N{t^i}}{(1-t)^2}$.

\item If $d=3$ or $4$,
we see that all components of the matrix $A$ are variables $x_i$.
\\Then $\displaystyle
H(R,t)=\frac{1+2t}{(1-t)^2}.$
\item If $d\geq 5$, $R$ cannot be nearly Gorenstein by Theorem 2.9.
\end{itemize}
\end{proof}
\end{thm}
\begin{rem}
We have already shown that
for any homogeneous affine semigroup ring $R$,
if it is nearly Gorenstein and $r(R)=2$,
then it is level,
the proof of Theorem 4.9 give the other proof in the case of $r(R)=\pd(R)=2$.
\end{rem}
For general homogeneous $d$-dimensional affine semigroup $R$,
nearly Gorensteinness does not imply the equation
$\displaystyle
H(R,t)=\frac{1+r(R)\sum_{i=1}^s{t^i}}{(1-t)^d}$.
Indeed, there are many counterexamples of $\pd(R) \geqq 4$.
\begin{ex}
$R=\kk[s,st^2,st^6,st^8,st^{11},st^{17},st^{23}]$ is nearly Gorenstein and
\[H(R,t)=\frac{1+5t+9t^2+6t^3+2t^4}{(1-t)^2}.\]
\end{ex}
However, for 2-dimensional homogeneous affine semigroup of projective dimension 3,
the following example exists.
\begin{ex}
$R=\kk[s,st^{2021},st^{2023},st^{4044},st^{6067}]$
is nearly Gorenstein and
\[H(R,t)=\frac{1+3\sum_{i=1}^{2022}t^i}{(1-t)^2}.\]
\end{ex}\section{nearly Gorenstein Stanley-Reisner rings}
The following example shows that
polarization does not necessarily preserve
nearly Gorensteinness.
Consider the polarization of Example 3.1.
\begin{ex}
Let $S=\QQ[x,y_1,y_2,y_3,z]$ be a polynomial ring with
$\deg x = \deg y_1= \deg y_2= \deg y_3= \deg z=1.$
Consider a homogeneous ideal $I=(xz,y_1 z,y_1 y_2 y_3)$ and define $R=S/I$,
then the graded minimal free resolution of $R$ is as follows.
\[0\rightarrow S(-3) \oplus S(-4) \xrightarrow {\left[
\begin{array}{cc}
-y_1 & 0 \\
x & -y_2y_3 \\
0 & z
\end{array}
\right]
}
S(-2)^{\oplus 2} \oplus S(-3) \rightarrow
 S \rightarrow R \rightarrow 0.\]
Thus $r(R)=2$ and $R$ is not level,
and $R$ is Cohen-Macaulay because $\dm R = \depth R = 3 > 0$.
Then, $R$ is not nearly Gorenstein by Theorem $2.9$.
\end{ex}
We recall some notation on simplicial complexes and Stanley-Reisner rings.
Let $\kk$ be a field and set $V=[n]=\{1,2,\cdots,n\}$.
A nonempty subset $\Delta$ of the power set $2^{V}$ of $V$
is called a \textit{simplicial complex} on $V$ if
$\{v\} \in \Delta$ for all $v \in V$, and $F \in \Delta$,
$H \subset F$ implies $H \in \Delta$.
For a face $F$ of $\Delta$, we put
$\link_\Delta(F) := \{ G \in \Delta \;;\; G\cup F \in \Delta, F\cap G = \emptyset  \}$.
This complex is called the \textit{link} of $F$.
The \textit{Stanley-Reisner ideal} of $\Delta$, denoted by $I_\Delta$
which is the squarefree monomial ideal generated by
\[\{x_{i_1}x_{i_2}\cdots x_{i_p} : 1\leq i_1 < \cdots < i_p \leq n,\;
\{x_{i_1},\cdots,x_{i_p}\} \notin \Delta \},\]
and $\kk[\Delta]=\kk[x_1,\cdots,x_n]/I_\Delta$ is called the
\textit{Stanley-Reisner ring} of $\Delta$.
Now we prepare some lemma.
\begin{lem}
Let $\Delta$ be a 0-dimensional simplicial complex consisting of $n \geq 2$ points,
and let $R=\kk[\Delta]=\kk[x_1,x_2,\cdots,x_n]/I_\Delta$ be the Stanley-Reisner ring of $\Delta$.
Then, the canonical ideal $I_R$ is generated by $\{x_1-x_2,x_1-x_3,\cdots,x_1-x_n\}$ and $R$ is
level and nearly Gorenstein.
\begin{proof}
Define an $R$-homomorphism $\phi:\kk[x_1]\oplus
\cdots \oplus \kk[x_n] \rightarrow \kk$ by
\begin{equation}
\begin{split} \notag
\phi( \;(f_1(x_1),f_2(x_2),\cdots,f_n(x_n))\; )=f_1(0)+f_2(0)+\cdots+f_n(0). \\
\text{for} \; \text{any}\;(f_1(x_1),f_2(x_2),\cdots,f_n(x_n)) \in
\kk[x_1]\oplus
\cdots \oplus \kk[x_n].
\end{split}
\end{equation}
First, we show that $\ker(\phi)$ is generated by
$(x_1,0,\cdots.0),(0,x_2,0,\cdots,0),\cdots, (0,\cdots,0,x_n)$ and
\;
$(1,-1,0,\cdots,0),(1,0,-1,0,\cdots,0),\cdots,(1,0,\cdots,0,-1)$
as an $R$-module.

For any $(f_1(x_1),f_2(x_2),\cdots,f_n(x_n)) \in \ker(\phi)$, we get the equality
$f_1(0)=-f_2(0)-\cdots-f_n(0)$.
We can write $f_i(x_i) = g_i(x_i)x_i + f_i(0)$ for any $1 \leq i \leq n$,
and putting $c_j = -f_j(0)$ for any $2 \leq j \leq n$, then we have
\begin{equation}\notag
\begin{split}
\begin{pmatrix} 
  f_1(x_1)  \\
  f_2(x_2)  \\
  \vdots  \\
  f_n(x_n) 
\end{pmatrix}
&=
\begin{pmatrix} 
  g_1(x_1)x_1 + c_2 + \cdots + c_n \\
  g_2(x_2)x_2 - c_2 \\
  \vdots  \\
  g_n(x_n)x_n - c_n
\end{pmatrix}
\\
&=g_1(x_1)
\begin{pmatrix} 
  x_1  \\
  0  \\
  \vdots  \\
  0 
\end{pmatrix}
+g_2(x_2)
\begin{pmatrix} 
  0  \\
  x_2  \\
  \vdots  \\
  0 
\end{pmatrix}
+\cdots
+g_n(x_n)
\begin{pmatrix} 
  0  \\
  \vdots  \\
  0 \\
  x_n 
\end{pmatrix}
\\
&+c_2
\begin{pmatrix} 
  1 \\
  -1 \\
  0 \\
  0 \\
  \vdots  \\
  0 
\end{pmatrix}
+c_3
\begin{pmatrix} 
  1  \\
  0  \\
  -1 \\
  0 \\
  \vdots  \\
  0
\end{pmatrix}
+\cdots
+c_n
\begin{pmatrix} 
  1  \\
  0 \\
  0 \\
  \vdots  \\
  0 \\
  -1
\end{pmatrix}.
\end{split}
\end{equation}
Thus, $\ker(\phi)$ is generated by
$(x_1,0,\cdots.0),\cdots, (0,\cdots,0,x_n)$
and
$(1,-1,0,\cdots,0),$\\
$(1,0,-1,0,\cdots,0),\cdots,(1,0,\cdots,0,-1)$
as an $R$-module.
Then, by [2, Section 5.7], we get
\[I_R=({x_1}^2, \cdots, {x_n}^2, x_1-x_2, x_1-x_3, \cdots,
x_1-x_n)=(x_1-x_2, x_1-x_3, \cdots,
x_1-x_n).\]
Thus, $R$ is level. We show that $R$ is nearly Gorenstein.
We define $g = (n-1)x_1-x_2-\cdots-x_n$.
Then, $g$ is a non-zero divisor of $R$.
Indeed, we assume $fg\equiv 0$ $(\text{mod}\;I_\Delta)$ for some $f \in R$.
Since $f \in R$, we can write $f(x_1,x_2,\cdots,x_n)\equiv x_1f_1(x_1)+x_2f_2(x_2)+\cdots+x_nf_n(x_n)+c$
$(\text{mod}\;I_\Delta)$, where $f_i \in k[x_i]$ for any $1 \leq i \leq n$ and $c \in \kk.$
Thus we have
\[fg\equiv(n-1)x_1(x_1f_1(x_1)+c)+\sum_{k=2}^n x_k(x_kf_k(x_k)+c)\equiv0 \;\; (\text{mod}\;I_\Delta).\]
From the definition of $I_\Delta$, we get the next equality in $\kk[x_1,x_2,\cdots,x_n]$.
\[(n-1)x_1(x_1f_1(x_1)+c)+\sum_{k=2}^n x_k(x_kf_k(x_k)+c)=0.\]
Thus, $x_if_k(x_i)+c = 0$ for any $1 \leq i \leq n$.
Therefore we get $c=0$ and $f(x_i)=0$ for any $1 \leq i \leq n$ and $f \equiv 0$ $(\text{mod}\;I_\Delta)$,
thus $g$ is a non-zero divisor of $R$.

Since $g$ is a non-zero divisor of $R$ and $g \in I_R$,
we have $\tr(\omega_R) = I_RI_R^{-1}$.
Here, for any $2 \leq i \leq n$ and for any $2 \leq k \leq n$, we get the following equality.
\begin{equation} \notag
\frac{(n-1)x_1}{g} (x_1-x_k)
=
\frac{(n-1){x_1}^2}{g}
=
\frac{x_1}{1},
\end{equation}

\begin{equation} \notag
\frac{x_i}{g} (x_1-x_k)
=
\begin{cases}
x_i & (i = k)\\
0 & (i \neq k).
\end{cases}
\end{equation}
Then we get $\displaystyle \frac{(n-1)x_1}{g},\frac{x_i}{g} \in I_R^{-1}$
for any $2 \leq i \leq n$
and $x_1,x_2,\cdots, x_n \in I_R{I_R}^{-1}=\tr(\omega_R)$. Therefore, $R$ is nearly Gorenstein.
\end{proof}
\end{lem}
\begin{lem}
Let $\Delta$ be a path with $n \geq 2$ edges and $n+1$ vertices,
and let $R=\kk[\Delta]=\kk[x_1,x_2,\cdots,x_n,x_{n+1}]/I_\Delta$ be the Stanley-Reisner ring of $\Delta$.
Then, the canonical ideal $I_R$ is generated by $\{x_1{x_2}^2+{x_2}^2x_3,x_2{x_3}^2+{x_3}^2x_4,\cdots,x_{n-1}{x_n}^2+{x_n}^2x_{n+1}\}$ and
$R$ is level and nearly Gorenstein.
\begin{proof}
Define an $R$-homomorphism
$\phi:\kk[x_1,x_2]\oplus \kk[x_2,x_3] \oplus
\cdots \oplus \kk[x_n,x_{n+1}] \rightarrow \kk[x_1]\oplus
\cdots \oplus \kk[x_n] \oplus \kk[x_{n+1}]$ by
\begin{equation} \notag
\begin{split}
\phi
\begin{pmatrix} 
  f_1(x_1,x_2)  \\
  f_2(x_2,x_3)  \\
  \vdots  \\
  f_{n-1}(x_{n-1},x_{n}) \\
  f_n(x_n,x_{n+1})
\end{pmatrix}
&=
\begin{pmatrix} 
  -f_1(x_1,0)  \\
  f_1(0,x_2)-f_2(x_2,0) \\
  \vdots  \\
  f_{n-1}(0,x_n)-f_n(x_n,0) \\
  f_n(0,x_{n+1})
\end{pmatrix}
. \\
\text{for any}
\begin{pmatrix} 
  f_1(x_1,x_2)  \\
  f_2(x_2,x_3)  \\
  \vdots  \\
  f_{n-1}(x_{n-1},x_{n}) \\
  f_n(x_n,x_{n+1})
\end{pmatrix}
& \in
\kk[x_1,x_2]\oplus \kk[x_2,x_3] \oplus
\cdots \oplus \kk[x_n,x_{n+1}].
\end{split}
\end{equation}
First, we show that $\ker(\phi)$ is generated by
$(x_1x_2,0,\cdots.0),\cdots,(0,\cdots,0,x_nx_{n+1})$
and \\
$(x_2,x_2,0,\cdots,0),(0,x_3,x_3,0,\cdots,0),\cdots,(0,\cdots,0,x_n,x_n)$
as an $R$-module.

For any $(f_1(x_1,x_2),f_2(x_2,x_3),\cdots,f_n(x_{n-1},x_n),f_n(x_n,x_{n+1})) \in \ker(\phi)$, we get the equality
\;
$f_1(x_1,0)=0$,
$f_{k-1}(0,x_k)=f_k(x_k,0)$ for any $2 \leq k \leq n$,
and $f_{n}(0,x_{n+1})=0.$
By $f_1(x_1,0)=0$, we can write
\begin{equation}
f_1(x_1,x_2)=x_2g_1(x_1,x_2).
\end{equation}
Indeed, since $f_1(x_1,x_2) \in \kk[x_2][x_1]$,
we can write $f_1(x_1,x_2)=x_2g_1(x_1,x_2)+c_1(x_1)$
and
$c_1(x_1) = f_1(x_1,0) = 0,$
thus $f_1(x_1,x_2)=x_2g_1(x_1,x_2).$
Next, by $f_{k-1}(0,x_k)=f_k(x_k,0)$, we have
\begin{equation}
f_k(x_k,x_{k+1})=x_{k+1}g_k(x_k,x_{k+1})+x_kg_{k-1}(0,x_k). \;\;(\text{for any} \;2 \leq k \leq n.)
\end{equation}
Indeed, since $f_k(x_k,x_{k+1}) \in \kk[x_{k+1}][x_k]$,
we can write $f_k(x_k,x_{k+1})=x_{k+1}g_k(x_k,x_{k+1})+c_k(x_k)$
and $c_k(x_k) = f_k(x_k,0) = f_{k-1}(0,x_k)=x_kg_{k-1}(0,x_k),$
thus $f_k(x_k,x_{k+1})=x_{k+1}g_k(x_k,x_{k+1})+x_kg_{k-1}(0,x_k).$
Finally, since $f_{n}(0,x_{n+1})=0$ and $g_1(x_1,x_2) \in \kk[x_1][x_2],\cdots
,g_{n-1}(x_{n-1},x_n) \in \kk[x_{n-1}][x_n]$, we can write
\begin{equation}
\begin{array}{l}
g_n(x_n,x_{n+1})=x_n h_n(x_n,x_{n+1}),\\
g_{k-1}(x_{k-1},x_k)=x_{k-1}h_{k-1}(x_{k-1},x_k)+g_{k-1}(0,x_k).\;\;(\text{for all} \;2 \leq k \leq n.)
\end{array}
\end{equation}
By equalities $(4),\;(5),\;(6),$ we have

\begin{equation}\notag
\begin{split}
\begin{pmatrix} 
  f_1(x_1,x_2)  \\
  f_2(x_2,x_3)  \\
  \vdots  \\
  f_{n-1}(x_{n-1},x_{n}) \\
  f_n(x_n,x_{n+1})
\end{pmatrix}
&=
\begin{pmatrix} 
  x_2g_1(x_1,x_2)  \\
  x_3g_2(x_2,x_3)+x_2g_1(0,x_2)  \\
  \vdots  \\
  x_ng_{n-1}(x_{n-1},x_{n})+x_{n-1}g_{n-2}(0,x_{n-1}) \\
  x_{n+1}g_n(x_n,x_{n+1})+x_ng_{n-1}(0,x_n)
\end{pmatrix}
\\
&=
\begin{pmatrix} 
  x_1x_2h_1(x_1,x_2)+x_2g_1(0,x_2)  \\
  x_2x_3h_2(x_2,x_3)+x_3g_2(0,x_3) + x_2g_1(0,x_2) \\
  \vdots  \\
  x_{n-1}x_nh_{n-1}(x_{n-1},x_n)+x_ng_{n-1}(0,x_n) + x_{n-1}g_{n-2}(0,x_{n-1}) \\
  x_{n}x_{n+1}h_n(x_n,x_{n+1})+x_ng_{n-1}(0,x_n)
\end{pmatrix}
\\
&=h_1(x_1,x_2)
\begin{pmatrix} 
  x_1x_2  \\
  0 \\
  0 \\
  \vdots  \\
  0 
\end{pmatrix}
+h_2(x_2,x_3)
\begin{pmatrix} 
  0  \\
  x_2x_3  \\
  0 \\
  \vdots  \\
  0 
\end{pmatrix}
+\cdots
+h_n(x_n,x_{n+1})
\begin{pmatrix} 
  0  \\
  \vdots  \\
  0\\
  0\\
  x_nx_{n+1}
\end{pmatrix}
\\
&+g_1(0,x_2)
\begin{pmatrix} 
  x_2 \\
  x_2 \\
  0 \\
  0 \\
  \vdots  \\
  0 
\end{pmatrix}
+g_2(0,x_3)
\begin{pmatrix} 
  0  \\
  x_3  \\
  x_3 \\
  0 \\
  \vdots  \\
  0
\end{pmatrix}
+\cdots
+g_{n-1}(0,x_n)
\begin{pmatrix} 
  0 \\
  0 \\
  \vdots  \\
  0 \\
  x_n \\
  x_n
\end{pmatrix}.
\end{split}
\end{equation}
Thus, $\ker(\phi)$ is generated by
$(x_1x_2,0,\cdots.0),(0,x_2x_3,0,\cdots,0),\cdots, (0,\cdots,0,x_nx_{n+1})$
and
\\
$(x_2,x_2,0,\cdots,0),(0,x_3,x_3,0,\cdots,0),\cdots,(0,\cdots,0,x_n,x_n)$
as an $R$-module.

Then, By [2, Section 5.7], we get
\begin{equation}\notag
\begin{split}
I_R&=({x_1}^2{x_2}^2, {x_2}^2{x_3}^2, \cdots, {x_n}^2{x_{n+1}}^2,
x_1{x_2}^2+{x_2}^2x_3, x_2{x_3}^2+{x_3}^2x_4, \cdots, x_{n-1}{x_n}^2+{x_n}^2x_{n+1})\\
&=(x_1{x_2}^2+{x_2}^2x_3, x_2{x_3}^2+{x_3}^2x_4, \cdots, x_{n-1}{x_n}^2+{x_n}^2x_{n+1}).
\end{split}
\end{equation}
Thus, $R$ is level. We show that $R$ is nearly Gorenstein.
We define $g = \sum_{k=1}^n {x_k}^2{x_{k+1}}^2$.
Then, in the same way as
Lemma 5.2,
we can show $g$ is a non-zero divisor of $R$.


Since $g$ is a non-zero divisor of $R$ and $g \in I_R$,
we have $\tr(\omega_R) = I_RI_R^{-1}$.
Here, for any $1 \leq k \leq n-1$, we get
\begin{equation} \notag
\frac{{x_1}^2}{g} (x_{k}{x_{k+1}}^2+{x_{k+1}}^2x_{k+2})
=
\begin{cases}
x_1 & (k=1)\\
0 & (k:\text{otherwise}),
\end{cases}
\end{equation}
\begin{equation} \notag
\frac{{x_{n+1}}^2}{g} (x_{k}{x_{k+1}}^2+{x_{k+1}}^2x_{k+2})
=
\begin{cases}
x_{n+1} & (k=n-1)\\
0 & (k:\text{otherwise}),
\end{cases}
\end{equation}
\begin{equation} \notag
\frac{\sum_{j=1}^n x_jx_{j+1}}{g} (x_{k}{x_{k+1}}^2+{x_{k+1}}^2x_{k+2})
=
\frac{{x_{k}}^2{x_{k+1}}^3+{x_{k+1}}^3{x_{k+2}}^2}{g}
=\frac{x_{k+1}}{1}.
\end{equation}

Then we get \[\frac{{x_1}^2}{g},\frac{{x_{n+1}}^2}{g},\frac{\sum_{j=1}^n x_jx_{j+1}}{g} \in I_R^{-1}\]
and $x_1,x_2,\cdots, x_n, x_{n+1} \in I_R{I_R}^{-1}=\tr(\omega_R)$. Therefore, $R$ is nearly Gorenstein.
\end{proof}
\end{lem}

Now, we can characterize nearly Gorensteinness of the Stanley-Reisner rings of low-dimensional simplicial complexes.

\begin{thm}\label{C}

(a) Let $\Delta$ be a 0-dimensional simplicial complex
and let $R=\kk[\Delta]$ be the Stanley-Reisner ring of $\Delta$.
Then $R$ is nearly Gorenstein and level.

\noindent(b) Let $\Delta$ be a 1-dimensional connected simplicial complex.
Then the following conditions are equivalent:

(1) $\Delta$ is nearly Gorenstein;

(2) $\Delta$ is Gorenstein on the punctured spectrum;

(3) $\Delta$ is locally Gorenstein
(i.e., $\kk[\link_\Delta(\{i\})]$ is Gorenstein for every $i \in V$);

(4) $\Delta$ is a path or a cycle.

\noindent(c) Every 1-dimensional nearly Gorenstein simplicial complex is almost Gorenstein and level.

\begin{proof}
(a) follows from Lemma 5.2. We show (b).
$(1) \Rightarrow (2)$ is known (see Remark 2.4).
$(2) \Rightarrow (3)$ follows from the isomorphism
$R_{x_i} \cong \kk[\link_\Delta(\{i\})][x_i,{x_i}^{-1}].$
We show $(3) \Rightarrow (4)$.
If we assume $(3)$,
then $\link_\Delta(\{i\})$ is one point or two points for every $i \in V$.
Thus, $\Delta$ is a path or a cycle.
Since a cycle is Gorenstein,
the implication $(4) \Rightarrow (1)$ follows from Lemma 5.3.
Next, we show (c).
For a 1-dimensional simplicial complex $\Delta$,
$\kk[\Delta]$ is almost Gorenstein
if and only if $\Delta$ is a tree or a ridge sum of cycles
(see \cite[Proposition 3.8]{MM}).
Since a path is level by Lemma 5.3,
1-dimensional nearly Gorenstein simplicial complex (a path or a cycle) is almost Gorenstein and level.
 \end{proof}

\end{thm}

\begin{rem}
Fix an integer $n \geq 2$.
Let $\Delta$ be a $n$-dimensional Cohen-Macaulay simplicial complex.
The same argument as $(1) \Rightarrow (2) \Rightarrow (3)$ in (b) of
Theorem 5.4 also holds.
Thus, if $\Delta$ is nearly Gorenstein,
then $\Delta$ is locally Gorenstein.
However, the converse does not hold in general.
\end{rem}

\begin{ex}
Let $\Delta$ be the standard triangulation
of the real projective plane ({see \cite[Figure 5.8]{BH}}).
Then $\mathbb{Q}[\Delta]$ is Cohen-Macaulay and locally Gorenstein.
However, this is not nearly Gorenstein because $\tr(\omega_{\mathbb{Q}[\Delta]})={\mathbf{m}}^2$.
Thus, this is also Gorenstein on the punctured spectrum.
Indeed, we put $S=\mathbb{Q}[x_1,\cdots,x_6]$ and
\[I_\Delta = (x_1x_2x_3,x_1x_2x_4,x_1x_3x_5,x_1x_4x_6,x_1x_5x_6,x_2x_3x_6,x_2x_4x_5,x_2x_5x_6,x_3x_4x_5,x_3x_4x_6).\]
Then the graded minimal free resolution of $\mathbb{Q}[\Delta]=S/I_\Delta$ is as follows.
\[0\rightarrow S(-5)^{\oplus 6} \xrightarrow {A
}
S(-4)^{\oplus 15} \rightarrow
 S(-3)^{\oplus 10} \rightarrow \mathbb{Q}[\Delta] \rightarrow 0.\]
 Thus we can check $\textrm{Ker}(\mathbb{Q}[\Delta] \otimes A)$ is generated
 by the following six column vectors by using $\mathtt{Macaulay2}$ ({\cite{M2}}).
 \[\left[
\begin{array}{cccccc}
{x_6}^2 \\
x_5x_6 \\
-x_4x_6 \\
-x_3x_6 \\
x_2x_6 \\
x_1x_6
\end{array}
\right],
\left[
\begin{array}{cccccc}
-x_5x_6 \\
-{x_5}^2 \\
x_4x_5 \\
x_3x_5 \\
x_2x_5 \\
x_1x_5
\end{array}
\right],
\left[
\begin{array}{cccccc}
-x_4x_6 \\
-x_4x_5 \\
{x_4}^2 \\
-x_3x_4 \\
-x_2x_4 \\
x_1x_4
\end{array}
\right],
\left[
\begin{array}{cccccc}
x_3x_6 \\
x_3x_5 \\
x_3x_4 \\
-{x_3}^2 \\
-x_2x_3 \\
x_1x_3
\end{array}
\right],
\left[
\begin{array}{cccccc}
x_2x_6 \\
-x_2x_5 \\
-x_2x_4 \\
x_2x_3 \\
{x_2}^2 \\
x_1x_2
\end{array}
\right],
\left[
\begin{array}{cccccc}
x_1x_6 \\
-x_1x_5 \\
x_1x_4 \\
-x_1x_3 \\
x_1x_2 \\
{x_1}^2
\end{array}
\right]
\in {\mathbb{Q}[\Delta]}^6\]

Then, we get $\tr(\omega_{\mathbb{Q}[\Delta]})={\mathbf{m}}^2$ by Theorem 2.9.
\end{ex}

\begin{Q}
Let $n \geq 2$ be an integer.
Is there any n-dimensional simplicial complex
such that it is nearly Gorenstein but not Gorenstein?
Are Gorensteinness and nearly Gorensteinness perhaps the same in this case?
\end{Q}

\renewcommand{\refname}{References}

\end{document}